\documentclass[a4paper,autoref]{lipics-v2021}\nolinenumbers
\usepackage{amsmath,amssymb,amsfonts,latexsym}
\usepackage{color,graphicx}
\usepackage{url} 
\usepackage{fancybox}
\usepackage{algorithm}
\usepackage[noend]{algpseudocode}
\usepackage{mathtools}
\usepackage{enumerate,xspace}
\usepackage{thm-restate}
\usepackage{siunitx} 
\hideLIPIcs

\newcommand{\eps}{\varepsilon}
\renewcommand{\epsilon}{\eps}
\newcommand{\etal}{\emph{et al.}\xspace}

\theoremstyle{plain}

\newenvironment{myquote}%
  {\list{}{\leftmargin=4mm\rightmargin=4mm}\item[]}%
  {\endlist}

\renewcommand{\leq}{\leqslant}
\renewcommand{\geq}{\geqslant}

\counterwithin{lemma}{section}
\counterwithin{theorem}{section}
\counterwithin{proposition}{section}
\counterwithin{claim}{section}
\counterwithin{observation}{section}
\counterwithin{corollary}{section}

\title{A Simple Construction of Tournaments with Finite and Uncountable Dichromatic Number}
 \funding{AS is supported by the  Dutch Research Council (NWO) through    Gravitation-grant NETWORKS-024.002.003.}

\author{Arpan Sadhukhan}{Department of Mathematics and Computer Science, TU Eindhoven, the Netherlands}{A.Sadhukhan@tue.nl}{}{}

\authorrunning{A.~Sadhukhan} 
 \Copyright{Arpan Sadhukhan}

\ccsdesc[100]{Graph Theory}

\keywords{dichromatic number, tournaments, sparse partition, grid coloring, uncountable, infinite graphs}

\begin{document}
\setcounter{page}{0}
\maketitle

\begin{abstract}
The dichromatic number $\chi(\vec{G})$ of a digraph $\vec{G}$ is the minimum number of colors needed to color the vertices $V(\vec{G})$ in such a way that no monochromatic directed cycle is obtained. In this note, for any $k\in \mathbb{N}$, we give a simple construction of tournaments with dichromatic number exactly equal to $k$. The proofs are based on a combinatorial lemma on partitioning a checkerboard which may be of independent interest. We also generalize our finite construction to give an elementary construction of a complete digraph of cardinality equal to the cardinality of $\mathbb{R}$ and having an uncountable dichromatic~number. Furthermore, we also construct an oriented balanced complete $n$-partite graph $\vec{K}^{(m)}_n$, such that the minimum number of colors needed to color its vertices such that there is no monochromatic directed triangle is greater than or equal to $nm/(n+2m-2)$.

\end{abstract}

\section{Introduction}
The dichromatic number $\chi(\vec{G})$ of a digraph $\vec{G}$ is the minimum number of colors needed to color the vertices $V(\vec{G})$ in such a way that no monochromatic directed cycle is obtained. This concept extends the notion of the chromatic number of a graph and was introduced by Neumann-Lara~\cite{NEUMANNLARA1982265}. Erd\H{o}s and Neumann-Lara~\cite{erdos1979problems} proved that there exist tournaments with dichromatic number $O(n/\log n)$, where $n$ is the number of vertices of the tournament. Unlike the chromatic number, determining a graph with a specific dichromatic number, even for tournaments, poses challenges in construction. In this paper, we provide a straightforward construction method for tournaments with a dichromatic number precisely equal to any given natural number $k$.~Although such construction is known in literature~\cite{NEUMANNLARA198483}, our construction is simpler and very different. The dichromatic number of digraphs, particularly tournaments, has been extensively studied, a few of those can be found in~\cite{HARUTYUNYAN20121823,NeumannLara1994The3A,NeumannLara2000DichromaticNC}. It is to be noted that the dichromatic number of tournaments is also related to the Erd\H{o}s-Hajnal conjecture~\cite{chudnovsky2016erd}. Let $K^{(m)}_n$ denote a complete balanced $n$-partite graph where each partition (or maximum independent set) has exactly $m$ vertices. Mohar~\etal~\cite{mohar_wu_2016} proved the existence of an orientation of $K^{(m)}_n$, whose dichromatic number is strictly greater than $\min\{nm/4\log(nm), n/2\}$, using a probabilistic approach. In this paper, we explicitly construct an oriented balanced complete $n$-partite graph $\vec{K}^{(m)}_n$, such that the minimum number of colors needed to color its vertices such that there is no monochromatic triangle is greater than or equal to $nm/(n+2m-2)$. Thus the constructed digraph $\vec{K}^{(m)}_n$ has dichromatic number greater than or equal to $nm/(n+2m-2)$. Additional insights and some recent developments on the dichromatic number are explored in~\cite{ellis2019cycle,harutyunyan2023colouring,steiner2023coloring}.

The dichromatic number for an infinite digraph $\vec{G}$ is the smallest cardinal $\kappa$ for which $V(\vec{G})$ can be colored with $\kappa$ many colors avoiding monochromatic
directed cycles. Within the framework of Zermelo-Fraenkel set theory with the Axiom of Choice (\emph{ZFC}), Galvin and Shelah~\cite{GALVIN1973167} constructed a tournament with cardinality $\aleph_1$ with uncountable dichromatic number, where $\aleph_1$ is the first uncountable cardinal. Soukup~\cite{soukup2018orientations} also proved similar and more generalized results in \emph{ZFC}. All these results relied on deep set theoretic theorems and knowledge. In fact Soukup~\cite{soukup2018orientations} highlighted the perceived non-triviality, within \emph{ZFC}, of discovering a single digraph with an uncountable dichromatic number. Even for finding tournaments of cardinality and dichromatic number $\aleph_1$, Soukup used a deep result of J.Moore. He posed the challenge of establishing an elementary proof for the existence of a tournament with cardinality $\aleph_1$ and an uncountable dichromatic number. In this note by assuming the continuum hypothesis along with \emph{ZFC}, we construct a tournament with cardinality $\aleph_1$ and uncountable dichromatic number. Although this does not answer Soukup's question in \emph{ZFC}, our construction distinguishes itself by its elementary nature, offering a clear and easily understandable approach with minimal reliance on set-theoretic prerequisites. Notably, our construction serves as a natural and intuitive extension of the finite tournament discussed in section~\ref{subsec:Finite construction}. Related results on dichromatic number of infinite graphs can be found in~\cite{DBLP:journals/jgt/Joo20,DBLP:journals/jgt/Komjath20}. 

\section{Tournaments and dichromatic number} \label{construction}
We first clarify a few basic notations that is used throughout the paper. All directed graphs in this paper refer to oriented graphs. Thus for a pair of node $u,v$ there can be at most one
directed edge. Let $G$  be a graph or a digraph, then $V(G)$ and $E(G)$ denote the vertex set and edge/arc set of $G$ respectively. $\vec{G}$ denotes an orientation of the undirected simple graph $G$. We say that $\vec{G}$ is acyclic iff it has no directed cycles. For a digraph $\vec{G}$ with $X \subseteq V(\vec{G})$, let $\vec{G}[X]$ denote the sub-digraph induced by $X$.

\subsection{Special sparse partition of a checkerboard} \label{subsec:Combinatorial lemma}

In this section, we prove a few combinatorial results on partitioning of a $n \times m$ checkerboard into some desired subsets. These results will be used to prove bounds on the dichromatic number of the constructed tournaments and balanced complete n-partite graphs in Sections~\ref{subsec:Finite construction} and~\ref{subsec:$n$-partite graphs}. Suppose we have an $n \times m$ checkerboard $G_{n \times m}$ with $n\times m$ cells, indexed by $(i,j)$ with $1\leq i \leq n$ and $j \leq m$. We give a total ordering of the cells of $G_{n \times m}$, we say that $(i_1,j_1)<(i_2,j_2)$ iff either $i_1<i_2$ or $i_1=i_2$ and $j_1<j_2$. A subset $S$ of the cells of $G_{n \times m}$ is said to be \emph{c-sparse} if for any two cells of $S$ with the same $y$-coordinate say $(a_1,b)$ and $(a_2,b)$, then $y\neq b$ and $(a_1,b)<(x,y)<(a_2,b)$ implies $(x,y) \notin S$. A partition $P$ of $G_{n \times m}$ is called \emph{c-sparse} if all the sets of the partition are \emph{c-sparse}. Let $\sigma(G_{n\times m})=\min\{|P|: P \; \text{is a \emph{c-sparse} partition of} \; G_{n \times m}\}$ denote the minimum number of disjoint \emph{c-sparse} sets needed to cover $G_{n \times m}$. Also define a subset $S$ of the cells of $G_{n \times m}$ to be \emph{weak-c-sparse} if for any two cells of $S$ with the same $y$-coordinate say $(a_1,b)$ and $(a_2,b)$, then $y\neq b$, $a_1<x<a_2$ and $(a_1,b)<(x,y)<(a_2,b)$ implies $(x,y) \notin S$. Let $R_i=\{(i,j):1 \leq j \leq m\}, C_j=\{(i,j):1 \leq i \leq n\}$ denote the $i^{th}$ row and the $j^{th}$ column of $G_{n \times m}$ respectively for all $1\leq i \leq n$ and $j \leq m$.

We prove the following theorem:
\begin{theorem} \label{sparse partition}
    $\sigma(G_{n\times n})$ equals the smallest natural number strictly greater than $\frac{n}{2}$.
\end{theorem}

Before we prove the theorem we proceed to prove a few necessary lemmas. Below is a trivial observation.

\begin{observation} \label{deletion}
    Let $S$ is a c-sparse set of $G_{n\times m}$. If we delete some rows and some columns of $G_{n\times m}$ to obtain a $n' \times m'$ checkerboard $G_{n'\times m'}$ with the ordering of the cells induced from $G_{n\times m}$. Then $S \cap G_{n'\times m'}$ is a c-sparse set of $G_{n'\times m'}$. 
\end{observation}

Now we prove that the maximum size of a \emph{c-sparse} set of $G_{n\times n}$ is upper bounded by $2n-1$.
\begin{lemma} \label{upper bound}
    Let $S$ be a c-sparse set of $G_{n\times n}$, then $|S|\leq 2n-1$.
\end{lemma}

\begin{proof}
  We proceed by induction on $n$. Clearly, for $n=1$ the theorem holds trivially. Now say the theorem holds for $n=k-1$. Let $S$ be a \emph{c-sparse} subset of $G_{k\times k}$ with cardinality greater than or equal to $2k$. Let $S'$ be a subset of $S$ with exactly $2k$ cells. Observe that $S'$ is trivially \emph{c-sparse}. Let $(i^*,j^*)$ be the smallest cell that belongs to $S'$. Then either $|S' \cap R_{i^*}|\leq 1$ or $|S' \cap C_{j^*}|\leq 1$, otherwise let $(i^*,j')$ and $(i',j^*)$ be two cells of $S'$ with $i' > i^*$ and $j' > j^*$. Then $(i^*,j^*)< (i^*, j') < (i', j^*)$ and $j' \neq j^*$ which is a contradiction to the fact that $S'$ is \emph{c-sparse}. Now we divide the proof into two cases.

\emph{Case 1}: $R_{i^*} \cap S'$ only consists of the element $(i^*,j^*)$.

Then $C_{j^*} \cap S'$ must have at least $2$ cells other than the cell $(i^*,j^*)$, otherwise $|(R_{i^*} \cup C_{j^*})\cap S'| \leq 2$. Now $G_{k\times k}\setminus (R_{i^*} \cup C_{j^*})$ is a $k-1 \times k-1$ checkerboard and let the ordering of the cells be induced from $G_{k\times k}$. So by Observation~\ref{deletion} we know that $S' \cap G_{k\times k}\setminus (R_{i^*} \cup C_{j^*})$ is a \emph{c-sparse} set of $G_{k\times k}\setminus (R_{i^*} \cup C_{j^*})$. Also $|S' \cap G_{k\times k}\setminus (R_{i^*} \cup C_{j^*}) |=2k-2$ which is a contradiction to the induction hypothesis. So we have $|C_{j^*} \cap S'| \geq 3$, hence by pigeon hole principle there exists  $j'' \leq k$ such that $|C_{j''} \cap S'| \leq 1$. So we have $|(R_{i^*} \cup C_{j''})\cap S'| \leq 2$, hence $G_{k\times k}\setminus (R_{i^*} \cup C_{j''})$ is a $k-1 \times k-1$ checkerboard and let the ordering of the cells be induced from $G_{k\times k}$. We have $|S' \cap G_{k\times k}\setminus (R_{i^*} \cup C_{j''}) | \geq 2k-2$, and since by Observation~\ref{deletion} we know that $S'$ is \emph{c-sparse} in $G_{k\times k}\setminus (R_{i^*} \cup C_{j''})$, we again have a contradiction to the induction hypothesis.

\emph{Case 2}: $C_{j^*} \cap S'$ only consists of the element $(i^*,j^*)$.

Then $R_{i^*} \cap S'$ must have at least $2$ cells other than the cell $(i^*,j^*)$, otherwise $|(R_{i^*} \cup C_{j^*})\cap S'| \leq 2$. Now $G_{k\times k}\setminus (R_{i^*} \cup C_{j^*})$ is a $k-1 \times k-1$ checkerboard and let the ordering of the cells be induced from $G_{k\times k}$. So by Observation~\ref{deletion} we know that $S' \cap G_{k\times k}\setminus (R_{i^*} \cup C_{j^*})$ is a \emph{c-sparse} set of $G_{k\times k}\setminus (R_{i^*} \cup C_{j^*})$. Also $|S' \cap G_{k\times k}\setminus (R_{i^*} \cup C_{j^*}) |=2k-2$ which is a contradiction to the induction hypothesis. So we have $|R_{i^*} \cap S'| \geq 3$, hence by pigeon hole principle there exists  $i'' \leq k$ such that $|R_{i''} \cap S'| \leq 1$. So we have $|(R_{i''} \cup C_{j^*})\cap S'| \leq 2$, hence $G_{k\times k}\setminus (R_{i''} \cup C_{j^*})$ is a $k-1 \times k-1$ checkerboard and let the ordering of the cells be induced from $G_{k\times k}$. We have $|S' \cap G_{k\times k}\setminus (R_{i''} \cup C_{j^*}) | \geq 2k-2$, and since by Observation~\ref{deletion} we know that $S'$ is \emph{c-sparse} in $G_{k\times k}\setminus (R_{i''} \cup C_{j^*})$, we again have a contradiction to the induction hypothesis.

Hence both the cases above cannot happen, thus finishing the proof.

  \end{proof}

Now we prove an upper bound on the size of \emph{c-sparse} set for the general $n \times m$ checkerboard.
\begin{theorem}\label{general upper bound}
    Let $S$ be a c-sparse set of $G_{n\times m}$, then $|S|\leq n+m-1$.
\end{theorem} 

\begin{proof}
Let $S$ be a \emph{c-sparse} subset of $G_{n\times m}$ with cardinality greater than or equal to $n+m$. Let $S'$ be a subset of $S$ with exactly $n+m$ cells. Observe that $S'$ is trivially \emph{c-sparse}. We divide the proof into three cases. 
\begin{itemize}
    \item \emph{Case 1: $n>m$}
    
Let $R'_j=R_j \cap S'$ for all $1\leq j \leq n$. Assume $|R'_{j_i}|\leq |R'_{j_{i+1}}|$ for all $1\leq i \leq n-1$. Now observe that if \[ \sum_{i=1}^{n-m} |R'_{j_i}| >  n-m. \] then clearly $|R'_{j_{n-m+1}}| \geq 2$ and hence \[ \sum_{i=1}^{n} |R'_{j_i}|= \sum_{i=1}^{n-m} |R'_{j_i}|+ \sum_{i=n-m+1}^{n} |R'_{j_i}| \geq  (n-m+1)+2m>  n+m. \] This is a contradiction as $|S'|=n+m$.
So we have~\[ \sum_{i=1}^{n-m} |R'_{j_i}| \leq  n-m. \]
Hence deleting the rows $R_{j_i}$ for all $1\leq i\leq n-m$ we obtain a $m \times m$ checkerboard $G_{m\times m}$ such that $S' \cap G_{m\times m}\geq n+m-(n-m)=2m$ which contradicts Lemma~\ref{upper bound}. Thus finishing the proof for \emph{Case 1}.

\item \emph{Case 2: $m>n$}

The proof of this case is exactly similar to the proof for \emph{Case 1}, the only difference is instead of rows $R_j$, we do the same argument with the columns $C_j$.

\item \emph{Case 3: $m=n$}

This case directly follows from Lemma~\ref{upper bound}.
\end{itemize}
Thus finishing the proof.
\end{proof}

Now consider the checkerboard $G_{n \times n}$, where $n$ is an odd natural number. Now let $D_i \subseteq G_{n \times n}=\{(x,y): (x-y)=i\}$ be the diagonals of $G_{n \times n}$. So clearly for $i>n-1$ and $i<-n+1$, $D_i$ is an empty set. So clearly $\bigcup\limits_{i=-n+1}^{n-1}D_i=G_{n \times n}$. Let $S_{2k}= D_{2k} \cup D_{2k+1} \cup D_{2k-n} \cup D_{2k-n-1}$ for $0 \leq k \leq (n-1)/2$. In the following lemma, we show that the subset $S_{2k}$ of $G_{n \times n}$, where $n$ is an odd natural number, is \emph{c-sparse}. For an illustration, see figure~\ref{fig:sparsepartition}.

\begin{lemma} \label{lower bound}
   For $0 \leq k \leq (n-1)/2$, the subset $S_{2k}$ of $G_{n \times n}$, where $n$ is an odd natural number, is c-sparse.
\end{lemma}
\begin{proof}
    Suppose $S_{2k}$ is not \emph{c-sparse}. Then there exists $(x_1,y), (x_2,y), (x', y') \in S_{2k}$ such that $(x_1,y)<(x',y')<(x_2,y)$ with $y' \neq y$. Observe that no two points in $D_i$ have the same y-coordinate. Now we divide the proof into two exhaustive sets of cases.

\emph{Case 1:}  One of the points between $(x_1,y)$ and $(x_2,y)$ lies in $D_{2k} \cup D_{2k+1}$ and the other lies in $D_{2k-n} \cup D_{2k-n-1}$.

Without loss of generality, assume $(x_1,y) \in D_{2k} \cup D_{2k+1}$ and $(x_2,y) \in D_{2k-n} \cup D_{2k-n-1}$. So $(x_1-y) \geq 2k$ and $(x_2-y) \leq 2k-n$, hence $x_1-x_2 \geq n$, which is a contradiction as they are points of $G_{n \times n}$.

\emph{Case 2:} Both the points $(x_1,y)$ and $(x_2,y)$ lies either in $D_{2k} \cup D_{2k+1}$ or in $D_{2k-n} \cup D_{2k-n-1}$.

Assume $(x_1,y) \in D_{2k}$ and $(x_2,y) \in D_{2k+1}$. So, $x_2=x_1+1$.  Now since $(x_1,y)<(x',y')<(x_2,y)$, either $(x',y')=(x_1, y')$ with $y'>y$ or $(x',y')=(x_1+1, y')$ with $y'<y$. In the former case $(x'-y')<(x_1-y)=2k$, also since $y'-y \leq n-1$,  $(x'-y')=(x_1-y')=(x_1-y-(y'-y)) \geq 2k-n+1$. So in this case $(x',y')$ cannot lie in $S_2k$. Now in the latter case when $(x',y')=(x_1+1, y')$ with $y'<y$, we have $(x'-y')=(x_1+1-y')>(x_1+1-y) \geq 2k+1$, so again in this case $(x',y')$ cannot belong to $S_{2k}$. Now the only other possibility that remains is that $(x_1,y) \in D_{2k-n-1}$ and $(x_2,y) \in D_{2k-n}$, then again using similar arguments it is easy to show that $(x',y')$ cannot belong to $S_{2k}$.

Thus, we have proved that the set $S_{2k}$ is \emph{c-sparse}.
\end{proof}

Now we proceed to prove Theorem~\ref{sparse partition}.

\begin{proof}[Proof of Theorem~\ref{sparse partition}]
    Let $n^*$ be the smallest integer strictly greater than $n/2$.
    From Lemma~\ref{upper bound} we know that the cardinality of any \emph{c-sparse} set is strictly less than $2n$. Hence if $P$ is a partition $G_{n \times n}$ into \emph{c-sparse} sets, then clearly $|P|> n^2/2n=n/2\geq n^*-1$.

On the other hand if $n$ is odd, we know from Lemma~\ref{lower bound} that for $0 \leq k \leq (n-1)/2$, the sets $S_{2k}$ are \emph{c-sparse}. Also trivially $\bigcup\limits_{i=0}^{(n-1)/2}S_{2k}=\bigcup\limits_{i=-n+1}^{n-1}D_i=G_{n \times n}$. So $P'= \{S_{2k}: 0 \leq k \leq (n-1)/2\}$, form a \emph{c-sparse} partition of $G_{n \times n}$. Also $|P'|=(n+1)/2=n^*$. Now it is easy to see that if there is a \emph{c-sparse} partition of $G_{n \times n}$ into $k$ sets, then trivially there exists a \emph{c-sparse} partition of $G_{n-1 \times n-1}$ into $k$ sets. So if $n$ is even, we already know that there is a \emph{c-sparse} partition of $G_{n+1 \times n+1}$ into $(n+2)/2$ sets, so there is a \emph{c-sparse} partition of $G_{n \times n}$ into $(n+2)/2=n^*$ sets. Thus finishing the proof.
\end{proof}

\begin{figure}
\begin{center}
\includegraphics[width=0.4\columnwidth]{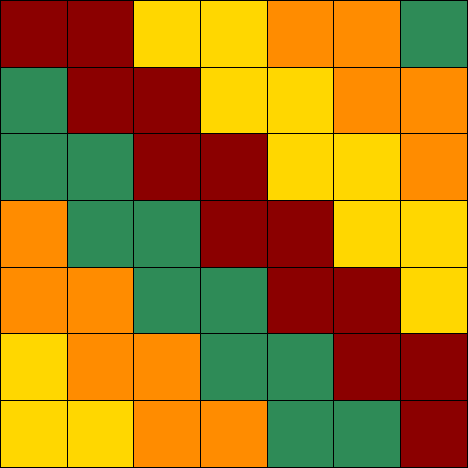}
\end{center}
\caption{Optimal \emph{c-sparse} partition of a $7 \times 7$ checkerboard, where each of the four color classes represents a \emph{c-sparse} set.}
\label{fig:sparsepartition}
\end{figure}

Now we prove an upper bound on the size of a \emph{weak-c-sparse} subset of $G_{n \times n}$.

\begin{lemma} \label{upper bound weak-c-sparse}
    Let $S$ be a weak-c-sparse set of $G_{n\times m}$, then $|S|\leq n+2m-2$.
\end{lemma}
\begin{proof}
Let $(i_j,j)$ denote the smallest element of $S \cap C_j$ for all $1\leq j \leq m$. Note that if $S \cap C_j= \emptyset$, then the smallest element does not exist. Let $S'=S \setminus \{(i_j,j): 1\leq j \leq n\}$, denote the subset obtained by deleting the smallest elements of every column that is present in $S$. Observe that $S'$ is trivially a \emph{weak-c-sparse} set of $G_{n\times m}$. Hence $S'$ is not \emph{c-sparse} only if there exist three cells  $(a_1,b), (x,y), (a_2,b)$ such that $y\neq b$, $x=a_1$ or $x=a_2$, $(a_1,b)<(x,y)<(a_2,b)$ and $(x,y) \in S'$. Now we divide the proof into two cases.

\emph{Case 1: $y\neq b$, $x=a_1$, $(a_1,b)<(x,y)<(a_2,b)$ and $(x,y) \in S'$}  
By construction of $S'$, we know that $(a_1,b)$ is not the smallest element of $S \cap C_b$. Let $(i_b,b)$ denote the smallest element of  $S \cap C_b$. So we have $y\neq b$, $(i_b,b), (x,y), (a_2,b)\in S$ and $(i_b,b)<(x,y)<(a_2,b)$ which is a contradiction as $S$ is a \emph{weak-c-sparse} set of $G_{n\times m}$ .

\emph{Case 2: $y\neq b$, $x=a_2$, $(a_1,b)<(x,y)<(a_2,b)$ and $(x,y) \in S'$}  
The proof is similar to the proof for \emph{Case 1}. Let $(i_b,b)$ and $(i_y, y)$ denote the smallest element of $S \cap C_b$ and $S \cap C_y$ respectively. If $i_y<a_1$ then we have $b\neq y$, $(i_y,y)<(a_1,b)< (x,y)$. If $a_1<i_y<a_2$, then we have $y\neq b$, $(a_1,b)<(i_y,y)<(a_2,b)$. Hence $i_y$ must be equal to $a_1$, otherwise, we get a contradiction as $S$ is a \emph{weak-c-sparse} set. Now again if $i_y=a_1$, then we have $b\neq y$ and $(i_b,b)<(i_y,y)<(a_2,b)$, which is again a contradiction as $S$ is a \emph{weak-c-sparse} set.

Now observe that by construction it is clear that $S' \cap R_1=\emptyset$. Also we have $S'$ is \emph{c-sparse} subset of $G_{n\times m}$, hence trivially it is also a \emph{c-sparse} subset of $G_{n-1\times m}$ which is obtained from $G_{n\times m}$ by deleting the first row $R_1$.
So by Lemma~\ref{general upper bound}, we have $|S'| \leq n+m-2$. Also by construction is easy to see that $|S|-|S'| \leq m$. Hence $|S|\leq n+2m-2$. Thus finishing the proof.
\end{proof}

\emph{Remark:} Observe that in a $n \times 1$ checkerboard, the union of all the cells is a \emph{c-sparse} as well as \emph{weak-c-sparse} subset of $G_{n\times 1}$. Hence the bounds in Lemma~\ref{general upper bound} and Lemma~\ref{upper bound weak-c-sparse} are tight. 

\subsection{Tournaments with any given fixed dichromatic number} \label{subsec:Finite construction}

Let $k \in \mathbb{N}$ be any fixed natural number. In this section, we give an easy construction of a tournament with dichromatic number exactly equal to $k$. Let $G_{2k-1 \times 2k-1}$ be a checkerboard of size $(2k-1)^2$. Let the cells $(i,j)$ be totally ordered with $(i_1,j_1)<(i_2,j_2)$ iff either $i_1<i_2$ or $i_1=i_2$ and $j_1<j_2$. Now we construct a tournament $\vec{T}_k$ with the number of vertices equal to $(2k-1)^2$ (Note that the notation $\vec{T}_k$ is slightly different from the standard notation of a tournament, where the subscript denotes the number of vertices but here it actually denotes the dichromatic number). 

\emph{Construction of $\vec{T}_k$:}
\begin{itemize}
    \item The $(2k-1)^2$ cells of $G_{2k-1 \times 2k-1}$ denote all the vertices of $\vec{T}_k$. The vertices are also totally ordered with the order induced from the total ordering of the cells of $G_{2k-1 \times 2k-1}$.
    
    \item If cells $(i_1,j_1),(i_2,j_2)$ lie in the same column of $G_{2k-1 \times 2k-1}$ (that is $j_1=j_2$), then the edge between $(i_1,j_1)$ and $(i_2,j_2)$ is directed from $(i_1,j_1)$ to $(i_2,j_2)$ iff $(i_1,j_1)<(i_2,j_2)$.

    \item If cells $(i_1,j_1),(i_2,j_2)$ does not lie in the same column of $G_{2k-1 \times 2k-1}$ (that is $j_1 \neq j_2$), then the edge between $(i_1,j_1)$ and $(i_2,j_2)$ is directed from $(i_1,j_1)$ to $(i_2,j_2)$ iff $(i_1,j_1)>(i_2,j_2)$.
\end{itemize}

Now we will prove that the dichromatic number of $\vec{T}_k$ is exactly equal to $k$. We define a subset of vertices of $\vec{T}_k$ to be \emph{c-sparse} iff the corresponding subset of cells of $G_{2k-1 \times 2k-1}$ is \emph{c-sparse}.

We state the following lemma which is very well known~\cite{bang2008digraphs} and easy to show.

\begin{lemma}\label{triangle}
    A tournament is acyclic iff it contains no directed triangles.
\end{lemma}

Now we observe that the \emph{c-sparse} sets of $\vec{T}_k$ are in one-to-one correspondence with the acyclic subtournaments of $\vec{T}_k$. 

\begin{lemma}\label{equivalence}
    A subset of vertices $V$ of $\vec{T}_k$ is c-sparse iff the induced subtournament $\vec{T}_{k}[V]$ is acyclic.
\end{lemma}
\begin{proof}
First, suppose the subset of vertices $V$ of $\vec{T}_k$ is not \emph{c-sparse}. Then there exists $(x_1,y), (x_2,y), (x', y') \in V$ such that $(x_1,y)<(x',y')<(x_2,y)$ with $y' \neq y$. So clearly the points $(x_1,y), (x_2,y), (x', y') \in V$ form a directed triangle in $\vec{T}_{k} [V]$ and hence $\vec{T}_{k}[V]$ is not acyclic.

Now suppose $\vec{T}_{k}[V]$ is not acyclic then by Lemma~\ref{triangle} we know that $\vec{T}_{k}[V]$ contains a directed triangle. Let $(x_1,y_1), (x_2,y_2), (x_3, y_3)$ form a directed triangle. Without loss of generality assume  $(x_3, y_3)>(x_2, y_2)$ and $(x_3, y_3)>(x_1, y_1)$. Also without loss of generality assume the vertex $(x_3, y_3)$ has an incoming edge from $(x_1,y_1)$ and an outgoing edge to $(x_2, y_2)$. Now by construction, we know that $y_1=y_3$ and $y_2 \neq y_3$. Since $(x_1,y_1), (x_2,y_2), (x_3, y_3)$ form a directed triangle, we also know that $(x_2, y_2)$ has an outgoing edge to $(x_1, y_1)$. Now since $y_2 \neq y_1$, by construction we have $(x_2, y_2)>(x_1, y_1)$. So we have $(x_1,y_1) <(x_2,y_2)< (x_3, y_3)$ with $y_2 \neq y_1=y_3$, which implies that $V$ is not \emph{c-sparse}. Thus finishing the proof.
\end{proof}

Now we have everything needed to show that the dichromatic number of $\vec{T}_k$ is exactly equal to $k$.

\begin{theorem}
    The dichromatic number of $\vec{T}_k$ is exactly equal to $k$.
\end{theorem}
\begin{proof}
    Suppose we color the vertices of $\vec{T}_k$ such that the vertices of each color class induce an acyclic subtournament. Then by Lemma~\ref{equivalence}, we know that the set of vertices of each color class is a \emph{c-sparse} set and also any \emph{c-sparse} set corresponds to an acyclic induced subtournament. So by Theorem~\ref{sparse partition} we know that $k$ colors are necessary and also sufficient to color the vertices of $\vec{T}_k$, such that the vertices of each color class form an induced acyclic subtournament. Hence the dichromatic number of $\vec{T}_k$ is exactly equal to $k$.
\end{proof}

\subsection{A lower bound for the dichromatic number of an oriented complete balanced $n$-partite graphs} \label{subsec:$n$-partite graphs}

Let $K^{(m)}_n$ denote a complete balanced $n$-partite graph where each partition (or maximum independent set) has exactly $m$. vertices. 

Let $G_{n \times m}$ be a checkerboard of size $nm$. Let the cells $(i,j)$ be totally ordered with $(i_1,j_1)<(i_2,j_2)$ iff either $i_1<i_2$ or $i_1=i_2$ and $j_1<j_2$. Now we construct an oriented complete balanced $n$-partite graph $\vec{K}^{(m)}_n$.

\emph{Construction:}
\begin{itemize}
    \item The $nm$ cells of $G_{n \times m}$, denote all the vertices of $\vec{K}^{(m)}_n$. Also, each row indicates one partition(or maximum independent set) of $\vec{K}^{(m)}_n$. The vertices are also totally ordered with the order induced from the total ordering of the cells of $G_{n \times m}$.
    
    \item If cells $(i_1,j_1),(i_2,j_2)$ lie in the same column of $G_{n \times m}$ (that is $j_1=j_2$), then the edge between $(i_1,j_1)$ and $(i_2,j_2)$ is directed from $(i_1,j_1)$ to $(i_2,j_2)$ iff $(i_1,j_1)<(i_2,j_2)$.

    \item If cells $(i_1,j_1),(i_2,j_2)$ does not lie in the same column of $G_{n \times m}$ (that is $j_1 \neq j_2$), then the edge between $(i_1,j_1)$ and $(i_2,j_2)$ is directed from $(i_1,j_1)$ to $(i_2,j_2)$ iff $(i_1,j_1)>(i_2,j_2)$.
\end{itemize}

We define a subset of vertices of $\vec{K}^{(m)}_n$ to be \emph{weak-c-sparse} iff the corresponding subset of cells of $G_{n \times m}$ is \emph{weak-c-sparse}. We begin with the simple observation below.

\begin{observation}\label{equivalence n-partite}
    For subset of vertices $V$ of $\vec{K}^{(m)}_n$ if the induced subdigraph $\vec{K}^{(m)}_n[V]$ has no directed triangle then $V$ is weak-c-sparse.
\end{observation}
\begin{proof}
    Suppose the subset of vertices $V$ of $\vec{K}^{(m)}_n$ is not \emph{weak-c-sparse}. Then there exists $(x_1,y), (x_2,y), (x', y') \in V$ such that $(x_1,y)<(x',y')<(x_2,y)$ with $x_1<x'< x_2$ and $y' \neq y$. So clearly by construction the points $(x_1,y), (x_2,y), (x', y') \in V$ form a directed triangle in the induced subdigraph $\vec{K}^{(m)}_n[V]$.
\end{proof}

\begin{theorem}\label{n-partite coloring}
The minimum number of colors needed to color the vertices of $\vec{K}^{(m)}_n$ such that there is no monochromatic directed triangle is greater than or equal to $nm/(n+2m-2)$.
    
\end{theorem}
\begin{proof}
    Suppose we color the vertices $\vec{K}^{(m)}_n$ such that the vertices colored with the same color do not induce a directed triangle. Now by Observation~\ref{equivalence n-partite}, we know that any subset of vertices of $\vec{K}^{(m)}_n$ which do not induce a directed triangle must be a \emph{weak-c-sparse} set of $G_{n \times m}$. Hence by Lemma~\ref{upper bound weak-c-sparse}, each color class must have cardinality less than or equal to $n+2m-2$. Hence the dichromatic number of $\vec{K}^{(m)}_n$ is greater than or equal to $nm/(n+2m-2)$.
\end{proof}

\emph{Remark:} It is trivial to observe that Theorem~\ref{n-partite coloring} implies that the dichromatic number of the constructed oriented graph $\vec{K}^{(m)}_n$ is greater than or equal to $nm/(n+2m-2)$.

\subsection{Uncountable tournaments with uncountable dichromatic number} \label{subsec: Infinite construction}

In this section, we assume the continuum hypothesis and \emph{ZFC}. Hence we know that $\aleph_1$ is the cardinality of $\mathbb{R}$. We will now construct a tournament on $\aleph_1$ vertices with dichromatic number $\aleph_1$, which is also a very natural generalization of the finite tournament $\vec{T}_k$ constructed in Section~\ref{subsec:Finite construction}.

For a point $p \in \mathbb{R}^2$, we denote by $x(p)$ and $y(p)$ to be the $x$-coordinate and $y$-coordinate of the point $p$ respectively.
Let $\vec{T}_{\aleph_1}$ be a tournament with vertex set being all the points on the plane $\mathbb{R}^2$(Note that cardinality of $\mathbb{R}^2$ is same as the cardinality of $\mathbb{R}$). Hence there is an edge between any two vertices of $T_{\aleph_1}$. Let the vertices of $T_{\aleph_1}$ denoted by $(i,j)$ be totally ordered with $(i_1,j_1)<(i_2,j_2)$ iff either $i_1<i_2$ or $i_1=i_2$ and $j_1<j_2$. Now we detail the construction of the uncountable tournament $\vec{T}_{\aleph_1}$ below.\\

 \emph{Construction of $\vec{T}_{\aleph_1}$:}
 \begin{itemize}
    \item Let $V(\vec{T}_{\aleph_1})=\mathbb{R}^2$.

\item For any two vertices $(i_1,j_1),(i_2,j_2)$ with $j_1=j_2$, the edge between $(i_1,j_1)$ and $(i_2,j_2)$ is directed from $(i_1,j_1)$ to $(i_2,j_2)$ iff $(i_1,j_1)<(i_2,j_2)$.

\item For any two vertices $(i_1,j_1),(i_2,j_2)$ with $j_1 \neq j_2$, then the edge between $(i_1,j_1)$ and $(i_2,j_2)$ is directed from $(i_1,j_1)$ to $(i_2,j_2)$ iff $(i_1,j_1)>(i_2,j_2)$.
    
\end{itemize}

We now state a standard well-known lemma in real analysis which easily follows from the separability property of $\mathbb{R}$. We will use the lemma in the proof of the main theorem.

\begin{lemma}\label{Interval}
     There do not exist uncountably many pairwise disjoint, non-empty open intervals of $\mathbb{R}$. 
\end{lemma}

\begin{theorem}
    The tournament $\vec{T}_{\aleph_1}$ has uncountable dichromatic number.
\end{theorem}
\begin{proof}
    Suppose for the sake of contradiction let $C=\{C_i: i\in \mathbb{N}\}$ denote a countable coloring of  $\vec{T}_{\aleph_1}$ such that there is no monochromatic directed cycle in $\vec{T}_{\aleph_1}$. Note that since all the points of the plane are vertices of $\vec{T}_{\aleph_1}$, $C$ can be also viewed as a coloring of all the points in the plane.  Let $V(C_i)$ denote the vertices of $\vec{T}_{\aleph_1}$ which are colored with the color $C_i$. Let $L_r$ denote the line $y=r$. Let $X=\{L_r: r\in\mathbb{R}\}$ be the set of all lines parallel to the x-axis. Define a function $f:X \rightarrow \mathbb{N}$, such that $f(L_r)=i$ iff $i= \min\{j \in \mathbb{N} : |V(C_j)\cap L_r| \geq 3  \}$. Since $C$ is a coloring of the whole plane all the points of the line $L_r$ must be colored using some color in $C$. Also, observe that $C$ is countable a set and $L_r$ is an uncountable set for all $r\in \mathbb{R}$. Hence for all $r\in \mathbb{R}$, there must exist at least three points in $L_r$ which get the same color. So the function $f$ is well-defined. Since $X$ is uncountable there exists $i^* \in \mathbb{N}$ such that the set $S=\{r \in \mathbb{R}:f(L_r)=i^*\}$ is uncountable. For any $r\in S$, let,
\[m_r=\inf_{p \in L_r \cap V(C_{i^*}) } x(p)\; \text{and} \; M_r=\sup_{p \in L_r \cap V(C_i^*) } x(p)\]. For $r\in S$, let $I_r$ denote the open interval $(m_r, M_r)$. Note that $m_r \neq M_r$, for all $r \in S$ as $|V(C_{i^*})\cap L_r| \geq 3$, hence $I_r$ is non-empty for all $r\in S$. 

\begin{claim}\label{disjoint}
     For $r\neq r' \in S$, the open interval $(m_r,M_r)$ is disjoint from $(m_{r'},M_{r'})$.
\end{claim}

\begin{claimproof}
If for some point $p_{r'} \in L_{r'} \cap V(C_{i^*})$, $x(p_{r'}) \in  (m_r,M_r)$. Then by definition from $m_r$ and $M_r$, there exists $q_r, q'_r \in L_{r} \cap V(C_{i^*})$, such that $x(p_{r'}) \in  (x(q_r), x(q'_r))$. Now since $x(q_r)<x(p_{r'})<x(q'_r)$ and $y(q_r)=y(q'_r)$, by construction we see that $p_{r'}, q_r, q'_r$ form a directed monochromatic triangle in $\vec{T}_{\aleph_1}$, which is a contradiction. Hence for any point $p_{r'} \in L_{r'} \cap V(C_{i^*})$, $x(p_{r'}) \notin  (m_r,M_r)$. So if $(m_r,M_r) \cap (m_{r'},M_{r'})$ is non-empty then $(m_r,M_r) \subseteq (m_{r'},M_{r'})$. Now since $|V(C_{i^*})\cap L_r| \geq 3$, there exists $p^*_r \in L_{r} \cap V(C_i^*)$, $x(p^*_r) \in  (m_r,M_r)$. So we have $x(p^*_r) \in (m_{r'},M_{r'})$. But by symmetry we also have for any point $p_r \in L_{r} \cap V(C_i^*)$, $x(p_r) \notin  (m_{r'},M_{r'})$, which gives a contradiction. Hence the open interval $(m_r,M_r)$ is disjoint from $(m_{r'},M_{r'})$. 
\end{claimproof}

Let $I=\{I_r: r\in S\}$ be a collection of non-empty open intervals of $\mathbb{R}$. Now since $S$ is uncountable, $I$ is uncountable. Also by Claim~\ref{disjoint}, we know that the intervals in $I$ are pairwise disjoint. Hence by Lemma~\ref{Interval} we have a contradiction. Thus finishing our proof.  

 \end{proof}

\section{Acknowledgements}
I thank Shivesh Roy and Partha Pratim Ghosh for some nice discussions.

\bibliography{references}

\newpage

\end{document}